\documentclass[11pt]{amsart}

\usepackage[T1]{fontenc}

\usepackage{amsmath,amssymb,amsfonts}
\usepackage{amsthm}
\usepackage[hidelinks]{hyperref}

\newtheorem{prop}{Proposition}[section]

\newtheorem{lemma}[prop]{Lemma}
\newtheorem{thm}[prop]{Theorem}

\theoremstyle{definition}
\newtheorem{rem}[prop]{Remark}

\newcommand{\Aut} {\mathrm{Aut}}
\newcommand{\Isom} {\mathrm{Isom}}

\newcommand{\End} {{\mathrm{End}}}

\newcommand{\ur}{\mathrm{ur}}
\newcommand{\ord}{\mathrm{ord}}

\def\BF{\mathbb{F}}
\def\BQ{\mathbb{Q}}
\def\BZ{\mathbb{Z}}

\def\CO{\mathcal{O}}
\def\CE{\mathcal{E}}

\def\Fp{\mathfrak{p}}

\begin{document}

\title[Divisibility of the coefficients of modular polynomials]{Divisibility of the coefficients of modular polynomials}

\author{Florian Breuer}
\address{School of Computer and Information Sciences, University of Newcastle, Callaghan, NSW 2308, Australia}
\email{florian.breuer@newcastle.edu.au}

\subjclass[2020]{11G07; 11G15}
\keywords{modular polynomials, isogenies, elliptic curves, complex multiplication}

\begin{abstract}
    Let $N>1$ and let $\Phi_N(X,Y)\in\mathbb{Z}[X,Y]$ be the modular polynomial which vanishes precisely at pairs of $j$-invariants of elliptic curves linked by a cyclic isogeny of degree $N$. In this note we study the divisibility of the coefficients of $\Phi_N(X+J, Y+J)$ for certain algebraic numbers $J$, in particular $J=0$ and other singular moduli. It turns out that these coefficients are highly divisible by small primes at which $J$ has supersingular reduction.
\end{abstract}

\maketitle

\section{Introduction}

Let $N$ be a positive integer and consider the classical modular polynomial $\Phi_N(X,Y)\in\BZ[X,Y]$
which vanishes precisely at pairs $(j_1,j_2)$ of $j$-invariants of elliptic curves linked by a cyclic $N$-isogeny. It has degree 
\[
\deg_X\Phi_N(X,Y) = \deg_Y\Phi_N(X,Y) := \psi(N) = N\prod_{p|N}\left(1+\frac{1}{p}\right).
\]
While the coefficients of $\Phi_N(X,Y)$ are notoriously large \cite{BPG2025, BP2024,  BS2010, Cohen1984} they are also highly divisible by small primes \cite{WANG202614}. 
Our first main result gives lower bounds on the $p$-orders of these coefficients.

\begin{thm}
	\label{thm:main}
    Let $N > 1$, and write 
    $\Phi_N(X,Y) = \sum_{0\leq i,j \leq \psi(N)}a_{i,j}X^iY^j.$
    Then for $i+j < \psi(N)$ the following hold.
    \begin{enumerate}
        \item If $2\nmid N$, then $v_2(a_{i,j}) \geq 15(\psi(N) -i -j)$.
        \item If $3\nmid N$, then $v_3(a_{i,j}) \geq 3(\psi(N) -i -j)$; moreover, 
        $v_3(a_{i,j}) \geq \lceil\frac{9}{2}(\psi(N) -i -j)\rceil$
        if $N\equiv 1 \bmod 3$.
        \item If $5\nmid N$ then
        $v_5(a_{i,j}) \geq 3(\psi(N) -i -j)$.
        \item If $p \geq 11$, $p \equiv 2 \bmod 3$ and $p\nmid N$, then $v_p(a_{i,j}) \geq 3\big(C_0(N,p)-i-j\big)$,
        where\
        \[
        C_0(N,p) := \ord_X (\Phi_N(X,0) \bmod p).
        \]
    \end{enumerate}
\end{thm}

When $p\leq 5$, this was conjectured by Wang in \cite{WANG202614}, who proved some related results and showed moreover that it suffices to prove Theorem \ref{thm:main} for prime $N$. The result has applications to the study of reduction types of elliptic curves, see \cite{Wang2025}.

The polynomials $\Phi_N(X,Y)$ have important applications in cryptography and computational number theory. Given finer bounds on the sizes of individual coefficients $a_{i,j}$, Theorem \ref{thm:main} may lead to tighter bounds on the Chinese Remainder Theorem primes required for CRT-based algorithms (e.g. \cite{BLS2012, BOS2016, CL2005, KD2025}) to compute $\Phi_N(X,Y)$. Theorem \ref{thm:main} can also be used as a quick sanity check for computations of modular polynomials for large $N$.

Modular polynomials have been computed for many values of $N$, see e.g. \cite{Sutherland-website} where one may download the coefficients of $\Phi_N(X,Y)$ for all $N \leq 400$ and many larger prime values of $N$. The files are rather large.

One can save space by only storing the factors of the coefficients of $\Phi_N(X,Y)$ not predicted by Theorem \ref{thm:main}. When $N=5$ (see Table \ref{tab:Phi5}) this reduces the number of decimal digits needed from 523 to 298, a $43\%$ saving. However, for larger $N$ the relative savings dwindle, for example when $N=101$ we only get a reduction from $6,383,216$ to $5,606,370$ decimal digits, a  $12\%$ saving. Alternatively, it may be useful to store the coefficients of $\Phi_N(X,Y)$ in partially factorized form, e.g. factoring up to prime divisors $p < 3N$.

\medskip

More generally, we study the coefficients of $\Phi_N(X+J_1, Y+J_2)$ for certain algebraic numbers $J_1$ and $J_2$, see Theorems \ref{thm:bigprimes} and \ref{thm:smallprimes} below. 
In particular, if $J_1 = J_2 = J$ is one of the 13 rational singular moduli we have the following theorem.

\begin{table}[h]
{\small
    \centering
    \begin{tabular}{|l|l|l|l|} \hline
        $J$ & $J-1728$ & $D$ & $n_p$ \\ \hline\hline 
        
         $0$ & $-2^6\cdot 3^3$ & $-3$ & 
         $n_3 = \left\{\begin{array}{ll}
            9/2 & : N \equiv 1 \bmod 3 \\
            3 & : N \equiv 2 \bmod 3 
        \end{array}\right.$ \\ \hline

        $2^{4} \cdot 3^{3} \cdot 5^{3}$ & $2^{4} \cdot 3^{3} \cdot 11^{2}$ & $-12$ & 
        $n_2 = 19/2$ \\
        & & & 
        $n_3 = \left\{\begin{array}{ll}
            9/2 & : N \equiv 1 \bmod 3 \\
            3 & : N \equiv 2 \bmod 3 
        \end{array}\right.$ \\ \hline
        
        $-2^{15} \cdot 3 \cdot 5^{3}$ & $- 2^{6} \cdot 3 \cdot 11^{2} \cdot 23^{2}$ & $-27$ & 
        $n_3 = \left\{\begin{array}{ll}
            4/3 & : N \equiv \pm1 \bmod 6  \\
             1/2 & : N \equiv \pm 2 \bmod 6
        \end{array}\right.$\\ \hline
        
        $2^{6} \cdot 3^{3}$ & $0$ & $-4$ & 
        $n_2 = \left\{\begin{array}{ll}
            10 & : N \equiv 1 \bmod 4  \\
             9 & : N \equiv 3 \bmod 4
        \end{array}\right.$\\ \hline
        
        $2^{3} \cdot 3^{3} \cdot 11^{3}$ & $2^{3} \cdot 3^{6} \cdot 7^{2}$ & $-16$ & 
        $n_2 = \left\{\begin{array}{ll}
            5 & : N \equiv 1 \bmod 4  \\
            9/2 & : N \equiv 3 \bmod 4
        \end{array}\right.$ \\ \hline
        
        $- 3^{3} \cdot 5^{3}$ & $- 3^{6} \cdot 7$ & $-7$ & $n_7 = 1$ \\ \hline
        
        $3^{3} \cdot 5^{3} \cdot 17^{3}$ & $3^{8} \cdot 7 \cdot 19^{2}$ & $-28$ & 
        $n_7 = 1$ \\ \hline
        
        $2^{6} \cdot 5^{3}$ & $2^{7} \cdot 7^{2}$ & $-8$ & $n_2 = 19/2$ \\ \hline
        
        $- 2^{15}$ & $- 2^{6} \cdot 7^{2} \cdot 11$ & $-11$ & 
        $n_{11} = 1$ \\ \hline
        
        $- 2^{15} \cdot 3^{3}$ & $- 2^{6} \cdot 3^{6} \cdot 19$ & $-19$ & 
        $n_{19} = 1$ \\ \hline
        
        $- 2^{18} \cdot 3^{3} \cdot 5^{3}$ & $- 2^{6} \cdot 3^{8} \cdot 7^{2} \cdot 43$ & $-43$ & $n_{43} = 1$ \\ \hline
        
        $- 2^{15} \cdot 3^{3} \cdot 5^{3} \cdot 11^{3}$ & $- 2^{6} \cdot 3^{6} \cdot 7^{2} \cdot 31^{2} \cdot 67$ & $-67$ & $n_{67} = 1$ \\ \hline
        
        $- 2^{18} \cdot 3^{3} \cdot 5^{3} \cdot 23^{3} \cdot 29^{3}$ & $- 2^{6} \cdot 3^{6} \cdot 7^{2} \cdot 11^{2} \cdot 19^{2} \cdot 127^{2} \cdot 163$ & $-163$ & 
        $n_{163} = 1$ \\ \hline     
    \end{tabular}
    \caption{Exceptional valuations of coefficients of $\Phi_N(X+J, Y+J)$ for singular moduli $J\in\BZ$.}
    \label{tab:rationalJ}
    }
\end{table}

\vspace{-5mm}
\begin{thm}\label{thm:rationalJ}
    Let $J \in\BZ$ be a rational singular modulus, i.e. $J=j(E)$ for an elliptic curve $E$ with complex multiplication by an imaginary quadratic order of discriminant $D < 0$ with class number $h(D)=1$. 
    Let $N>1$ and write $\Phi_N(X+J, Y+J) = \sum_{0\leq i,j \leq \psi(N)}a_{i,j}X^iY^j.$

    Suppose $p\nmid N$ and $\left(\frac{D}{p}\right) \neq 1$. Then 
    \[
    v_p(a_{i,j}) \geq n_p\big(C_J(N,p) - i-j\big)\quad\text{for all $i+j < C_J(N,p)$.}
    \]
    Here $C_J(N,p) = \ord_X\big(\Phi_N(X+J,J) \bmod p\big)$ and
    $n_p$ is given by
    \[
    n_p = \left\{
    \begin{array}{ll}
        15 & \text{if $p|J$ and $p=2$} \\
        6 & \text{if $p|J$ and $p=3$} \\
        3 & \text{if $p|J$ and $p\geq 5$}  \\
        2 & \text{if $p|(J-1728)$ and $p\geq 5$} \\
        1 & \text{otherwise,}
    \end{array}
    \right.
    \]
    except for the special cases listed in Table \ref{tab:rationalJ}.
\end{thm}

Theorem \ref{thm:main} is just the $J=0$ case of Theorem \ref{thm:rationalJ}.

By Proposition \ref{ordinary} below, we only get divisibility conditions of the form in Theorem \ref{thm:rationalJ} for primes $p$ at which the reduction of $E$ is supersingular, i.e. when $\displaystyle\left(\frac{D}{p}\right)\neq 1$.

In these cases, we have an explicit expression for $C_J(N,p)$ in terms of the theta series of the quaternion order $\End_{\bar{\BF}_p}(E)$; we get non-trivial divisibility results only for certain primes $p < |D|N$, see Proposition~\ref{supersingular}.

Computations show that the values of $n_p$ given in Theorem \ref{thm:rationalJ} are optimal, except in the cases $D=-12$ and $D=-27$, where we expect the true values to be
\begin{align*}
    & n_3 = \left\{\begin{array}{ll}
            5 & : N \equiv 1 \bmod 3 \\
            3 & : N \equiv 2 \bmod 3 
        \end{array}\right. 
        & &  \text{when $D=-12$ and} \\
    & n_3 = \left\{\begin{array}{ll}
            3/2 & : N \equiv \pm1 \bmod 6 \\
            1 & : N \equiv \pm2 \bmod 6 
        \end{array}\right. 
        & & \text{when $D=-27$.}
\end{align*}
Theorems \ref{thm:main} and \ref{thm:rationalJ} are a variation on the theme of differences of singular moduli pioneered by Gross and Zagier, see \cite{Campagna2023, Dorman1988, GZ1985, LV2015}.

\begin{table}[h]
\begin{tabular}{l}
$a_{0, 0} = \mathbf{2^{90} \cdot 3^{18} \cdot 11^{9}} \cdot 5^{3}$  \\
$a_{1, 0} = \mathbf{2^{75} \cdot 3^{15} \cdot 11^{6}} \cdot 2^{2} \cdot 3 \cdot 5^{3} \cdot 31 \cdot 1193$  \\
$a_{1, 1} = \mathbf{2^{60} \cdot 3^{12} \cdot 11^{3}} \cdot -1 \cdot 2^{2} \cdot 3 \cdot 26984268714163$  \\
$a_{2, 0} = \mathbf{2^{60} \cdot 3^{12} \cdot 11^{3}} \cdot 3 \cdot 5^{2} \cdot 13^{2} \cdot 3167 \cdot 204437$  \\
$a_{2, 1} = \mathbf{2^{45} \cdot 3^{9}} \cdot 2^{2} \cdot 3 \cdot 5^{4} \cdot 53359 \cdot 131896604713$  \\
$a_{3, 0} = \mathbf{2^{45} \cdot 3^{9}} \cdot 2^{3} \cdot 5^{2} \cdot 31 \cdot 1193 \cdot 24203 \cdot 2260451$  \\
$a_{2, 2} = \mathbf{2^{30} \cdot 3^{6}} \cdot 3^{2} \cdot 5^{4} \cdot 7 \cdot 13 \cdot 1861 \cdot 6854302120759$  \\
$a_{3, 1} = \mathbf{2^{30} \cdot 3^{6}} \cdot -1 \cdot 2 \cdot 3 \cdot 5^{3} \cdot 327828841654280269$  \\
$a_{4, 0} = \mathbf{2^{30} \cdot 3^{6}} \cdot 3 \cdot 5 \cdot 13^{2} \cdot 3167 \cdot 204437$  \\
$a_{3, 2} = \mathbf{2^{15} \cdot 3^{3}} \cdot 2^{2} \cdot 3 \cdot 5^{3} \cdot 2311 \cdot 2579 \cdot 3400725958453$  \\
$a_{4, 1} = \mathbf{2^{15} \cdot 3^{3}} \cdot 2^{5} \cdot 3 \cdot 5^{3} \cdot 12107359229837$  \\
$a_{5, 0} = \mathbf{2^{15} \cdot 3^{3}} \cdot 2^{2} \cdot 3 \cdot 5 \cdot 31 \cdot 1193$  \\
$a_{3, 3} = -1 \cdot 2^{2} \cdot 5^{2} \cdot 11 \cdot 17 \cdot 131 \cdot 1061 \cdot 169751677267033$ \\
$a_{4, 2} = 3 \cdot 5^{3} \cdot 167 \cdot 6117103549378223$ \\
$a_{5, 1} = -1 \cdot 2 \cdot 3 \cdot 5^{2} \cdot 1644556073$ \\
$a_{6, 0} = 1$ \\
$a_{4, 3}  = 2^{5} \cdot 3 \cdot 5^{2} \cdot 197 \cdot 227 \cdot 421 \cdot 2387543$ \\
$a_{5, 2}  = 2^{5} \cdot 5^{2} \cdot 13 \cdot 195053$ \\
$a_{4, 4}  = 2^{3} \cdot 5^{2} \cdot 257 \cdot 32412439$ \\
$a_{5, 3}  = -1 \cdot 2^{2} \cdot 3^{2} \cdot 5 \cdot 131 \cdot 193$ \\
$a_{5, 4}  = 2^{3} \cdot 3 \cdot 5 \cdot 31$ \\
$a_{5, 5}  = -1$ \\
\end{tabular}
\caption{Coefficients of $\Phi_5(X,Y) = \sum_{i,j}a_{i,j}X^iY^j$ with factors predicted by Theorem \ref{thm:main} in bold}
\label{tab:Phi5}
\end{table}

\section{The numbers $C_J(N,\pi)$}
From now on, let $K$ be a complete valued field of characteristic zero with valuation $v$, uniformizer $\pi$, ring of integers $A$ and algebraically closed residue field $A/\pi$ of characteristic $p$.

Let $E/K$ be an elliptic curve with good reduction, let $J = j(E)$ and $\CO_{J,\pi} = \End_{A/\pi}(E)$. Then 
\[
C_J(N,\pi) := \ord_X\big(\Phi_N(X+J, J) \bmod \pi\big)
\] 
counts the number of cyclic $N$-isogenies of $E$ which reduce to endomorphisms modulo $\pi$. This depends crucially on whether the reduced elliptic curve $E_{A/\pi}$ is ordinary or supersingular.

\begin{prop}\label{prop:fullpsi}
    Suppose $p\nmid N$. 
    We have $C_0(N,p) = \psi(N)$ for $p=2,3,5$; 
    $C_{1728}(N,p) = \psi(N)$ for $p=2,3,7$ and
    $C_5(N,13) = \psi(N)$.
\end{prop}

\begin{proof}
    The cases $J$ and $p$ in the statement are precisely those where $J$ is the only supersingular invariant in characteristic $p$. All roots of $\Phi_N(X, J) \bmod \pi$ correspond to elliptic curves isogenous to $E_{A/\pi}$ and are thus again supersingular, thus all $\psi(N)$ roots of $\Phi_N(X+J, J)$ reduce to $0 \bmod \pi$.
\end{proof}

\begin{prop}\label{ordinary}
    Suppose $p\nmid N$. Suppose $E$ has ordinary reduction, then $\CO_{J,\pi}$ is an order of discriminant $D$ in an imaginary quadratic field. 
    \begin{enumerate}
        \item  We have 
        \[
        C_J(N,\pi) \leq \#I_{\mathrm{cyc}}(N,\CO_{J,\pi})
        :=
        \#\{\mathfrak n\subset \CO_{J,\pi} :
                \CO_{J,\pi}/\mathfrak n\cong \BZ/N\BZ\},
        \]
        with equality if $\CO_{J,\pi}$ is a principal ideal domain.
        
        \item If $E/K$ also has complex multiplication (necessarily by an order of discriminant $Dp^{2m}$ for some $m \geq 0$), then we find that
        \[
        C_J(N,\pi) = C_J(N,0) := \ord_{X}\big(\Phi_N(X+J, J) \in K[X]\big).
        \]
    \end{enumerate}
\end{prop}

In case (2), $a_{C_J(N,\pi),0}$ is the first non-zero coefficient of $\Phi_N(X+J, J)$ and 
        $v(a_{C_J(N,\pi),0}) = 0$, so we get no non-trivial divisibility relations.

\begin{proof}
	Part (1) follows because $C_J(N,\pi)$ equals the number of principal ideals $\mathfrak{n} \in 
	I_{\mathrm{cyc}}(N,\CO_{J,\pi})$. 
	
	If $E/K$ has complex multiplication, then $\End_{\bar K}(E)$ and
	$\CO_{J,\pi}$ differ only at the prime $p$, or equivalently become equal
	after tensoring with $\BZ[1/p]$. Since $p\nmid N$, this does not affect
	cyclic $N$-isogenies. Part (2) follows.
\end{proof}

\begin{prop}\label{supersingular}
    Let $p\nmid N$. Suppose $E$ has supersingular reduction, then $\CO_{J,\pi}$ is a maximal order in the quaternion algebra ramified exactly at $p$ and $\infty$. 
    \begin{enumerate}
        \item We have
        \[
            C_J(N,\pi) = \frac{1}{\#\CO_{J,\pi}^*}\sum_{d^2|N}\mu(d)\#\{f\in\CO_{J,\pi} \; |\; \mathrm{nrd}(f) = N/d^2 \}.
        \]
        The cardinalities in the above sum are coefficients of the theta series associated to $\CO_{J,\pi}$.

        \item Now suppose $E/K$ has complex multiplication by a quadratic imaginary order $\CO_D$ of discriminant $D<0$ and $C_J(N,\pi) > C_J(N,0)$. Then $p < |D|N$.

    \end{enumerate}
\end{prop}

\begin{proof}
    For relevant facts about orders in quaternion algebras, see \cite[\S41-42]{Voight2021}.
    The quantity $\#\CO_{J,\pi}^* \cdot C_J(N,\pi)$ counts the total number of cyclic endomorphisms of $E$ over $A/\pi$ of degree $N$, which correspond to elements of reduced norm $N$ in $\CO_{J,\pi}$ with cyclic quotient.
    Counting all elements of reduced norm $N$ in $\CO_{J,\pi}$, not just those with cyclic quotient, gives
    \[
    \#\{f\in\CO_{J,\pi} \; |\; \mathrm{nrd}(f) = N \} = \#\CO_{J,\pi}^* \sum_{d^2|N} C_J(N/d^2,\pi)
    \]
    and part (1) now follows by M\"obius inversion.

    Now suppose the hypothesis of (2) holds. 
    The first non-zero coefficient $a_{C_J(N,0),0}$ of $\Phi_N(X+J,J)$ is a product of the form
    \[
    a_{C_J(N,0),0} = \prod_{\tilde{E}\to E}(j(\tilde{E})-J)
    \]
    where $\tilde{E}$ ranges over elliptic curves linked to $E$ by a cyclic $N$-isogeny, but for which $j(\tilde{E})\neq J$.    
    By assumption, this product reduces to 0 modulo $\pi$, so for one of these elliptic curves we have $j(\tilde{E})\neq j(E)$ and $\tilde{E}\cong E \bmod \pi$.
    This $\tilde{E}$ has complex multiplication by an order $\CO_{Df^2}$ of discriminant $Df^2$ for some $f|N$. 
    
    If $p$ divides the conductor of $\CO_D$ then $p < |D|N$ is clear. Otherwise, by \cite[Prop 2.2.]{LV2015}, the orders $\CO_D$ and $\CO_{Df^2}$ embed optimally into $\CO_{J,\pi}$. The result now follows from \cite[Thm. 2']{Kaneko1989}.
\end{proof}

\begin{rem}
    Theorems \ref{thm:main} and \ref{thm:rationalJ} give lower bounds on the absolute value of the first non-zero coefficient $a_{C_J(N,0),0}\in\BZ$. Combined with the upper bound on the size of the coefficients of $\Phi_N(X,Y)$ from \cite{BPG2025}, one may obtain upper bounds on certain averages of the $C_J(N,p)$'s. 

    For example, if $N$ is odd and $C_0(N,0)=0$ one can show
    \[
    \sum_{\substack{p < 3N \\ p\nmid N}} C_0(N,p)\log p \leq 2\psi(N)(\log N - \lambda_N + 8.2),
    \]
    where
    \[
    \lambda_N := \prod_{p^n\|N}\frac{p^n-1}{p^{n-1}(p^2-1)}\log p = O(\log\log N).
    \]  
\end{rem}

\section{Proof of the main results}
\subsection{General results}\label{sec:setup}

We have the following general result, which implies Theorems \ref{thm:main} and \ref{thm:rationalJ} when $p\geq 5$.

\begin{thm}\label{thm:bigprimes}
    Let $E_1, E_2/K$ be elliptic curves with good reduction and $J_1 = j(E_1), \; J_2 = j(E_2)\in A$.
    Let $N > 1$ with $p\nmid N$. For $k\in\{1,2\}$, let 
    \[
    n_k = \frac{1}{2}\#\Aut_{A/\pi}(E_k) = \left\{
    \begin{array}{ll}
        12 & \text{if $v(J_k) > 0$ and $p=2$} \\
        6 & \text{if $v(J_k) > 0$ and $p=3$} \\
        3 & \text{if $v(J_k) > 0$ and $p\geq 5$} \\
        2 & \text{if $v(J_k - 1728) > 0$ and $p\geq 5$} \\
        1 & \text{if $v(J_k) = v(J_k-1728) = 0$.}
    \end{array}
    \right.
    \]
    Then the coefficients of $\Phi_N(X+J_1, Y+J_2) = \sum_{0\leq i,j \leq \psi(N)}a_{i,j}X^iY^j\in A[X,Y]$ satisfy
    \[
    v(a_{i,j}) \geq n_1 C_{J_1,J_2}(N,\pi)- n_1 i-n_2 j = n_2 C_{J_2,J_1}(N,\pi)- n_1 i-n_2 j,
    \]
    for all $i,j < \psi(N)$.
    Here 
    \[
    C_{J_1,J_2}(N,\pi) := \ord_{X}\big(\Phi_N(X+J_1, J_2) \bmod \pi\big).
    \]
\end{thm}

In residue characteristics $p=2$ or $3$ the coefficients $a_{i,j}$ typically have larger valuations than those obtained in Theorem \ref{thm:bigprimes}. In the special case $E_1=E_2 = E$, we use V\'elu's formulae to obtain the following technical result. 
Suppose $E$ is defined by a minimal Weierstrass equation over $A$ with good reduction
\begin{equation}\label{eq:WE}
E : y^2 + a_1xy + a_3y = x^3 + a_2x^2 + a_4x + a_6
\end{equation}
and define as usual the associated quantities in $A$:
\begin{align*}
    b_2 & = a_1^2+4a_2, & b_4 & = a_1a_3 + 2a_4, & b_6 & = a_3^2 + 4a_6 \\
    c_4 & = b_2^2 - 24b_4, & c_6 & = -b_2^3 + 36b_2b_4 - 216 b_6, 
    & \Delta & = (c_4^3-c_6^2)/1728.
\end{align*}
Then $j(E) = c_4^3/\Delta$ and $v(\Delta)=0$. 

For $N > 1$, we now define the following polynomials in $K[x_0,x_1,x_2,x_3,y_0]$. If $N$ is odd, then
\begin{align}\label{eq:tw_odd}
    t & := 6x_2 + b_2x_1 + \left(\frac{N-1}{2}\right)b_4, \\
    w & := 10x_3 + 2b_2x_2 + 3b_4x_1 + \left(\frac{N-1}{2}\right)b_6, \nonumber
\end{align}
whereas, if $N$ is even, we define
\begin{align}\label{eq:tw_even}
    t & := 6x_2 + b_2x_1 + \left(\frac{N-2}{2}\right)b_4 +
    3x_0^2 + 2a_2x_0 + a_4 - a_1y_0 \\
    w & := 10x_3 + 2b_2x_2 + 3b_4x_1 + \left(\frac{N}{2}\right)b_6
    + 7x_0^3 + (b_2+2a_2)x_0^2 + (2b_4+a_4)x_0 - a_1x_0y_0.\nonumber
\end{align}
Finally, define
\begin{align}\label{eq:g}
g & := [(c_4+240t)^3c_6^2 - c_4^3(c_6 + 504b_2t + 6048w)^2]/1728 \in K[x_0,x_1, x_2,x_3, y_0] \\
n_v & = v(g) := \max \{n \;|\; g \in \pi^nA[x_0,x_1,x_2,x_3,y_0]\}.\nonumber
\end{align}

\begin{thm}\label{thm:smallprimes}
    Suppose $p=2$ or $3$.
    Let $E/K$ be an elliptic curve with good reduction and $J = j(E)\in \pi A$.
    Let $N > 1$ with $p\nmid N$. Then the coefficients of $\Phi_N(X+J, Y+J) = \sum_{0\leq i,j \leq \psi(N)}a_{i,j}X^iY^j\in A[X,Y]$ satisfy
    \[
    v(a_{i,j}) \geq n_v\big(C_J(N,\pi)-i-j\big)
    \]
    for all $i+j < C_J(N,\pi)$, where $n_v$ is defined in (\ref{eq:g}).

    Furthermore, when $p=2$ then $n_v$ only depends on $N \bmod 4$ and when $p=3$, $n_v$ only depends on $N \bmod 6$.
\end{thm}

\subsection{Polynomial interpolation}

A natural approach to proving Theorem \ref{thm:main} is the following. Set $Y=0$ to obtain
\[
\Phi_N(X,0) = \sum_{i=0}^{\psi(N)}a_{i,0}X^i = \prod_{E\to E'}(X - j(E')),
\]
where $E$ is an elliptic curve with $j(E)=0$ and the product ranges over elliptic curves $E'$ linked to $E$ via a cyclic isogeny of degree $N$. These are special elliptic curves with complex multiplication by orders in $\BQ(\sqrt{-3})$, so applying methods from e.g. \cite{Campagna2021}, \cite{GZ1985},  or applying V\'elu's formulae \cite{Velu1971} to a suitable model of $E$, gives lower bounds on the $v(j(E))$ which imply Theorem \ref{thm:main} for the coefficients $a_{i,0}$.

To study the coefficients $a_{i,j}$ in general, we need to specialize $Y=y_k=j(\CE_k) - j(E)$ for suitable {\em deformations} $\CE_k$ of the elliptic curve $E$ and then extract our coefficients via polynomial interpolation. The following lemma is key. 

\begin{lemma}\label{lem:interpolation}
    Let $f(Y) = a_0 + a_1Y + \ldots + a_dY^d \in K[Y]$.
    Fix $n\in\BZ$ and
    let $y_0, y_1,\ldots, y_d\in K$ be such that
    \begin{enumerate}
        \item $v(y_0)=v(y_1)=\cdots=v(y_d)=n$,
        \item $v(y_k-y_l) = n$ for all $k\neq l$.
    \end{enumerate}
    Then 
    \[
    v(a_j)  \geq \min_{0\leq k \leq d} v(f(y_k)) - nj
    \quad \text{for all $j=0,1,2,\ldots, d$}.
    \]
\end{lemma}

Conversely, if $v(a_j) \geq B - nj$ for all $j$, then clearly $v(f(y_k)) \geq B$.

\begin{proof}
    We solve for the coefficients $a_j$ in the linear system
    \[
    a_0 + a_1y_k + \cdots + a_dy_k^d = f(y_k), \quad k=0,1,2,\ldots d.
    \]
    By Cramer's rule, we get $a_j = \frac{M_j}{V}$, where $V = \det(y_k^i)_{0\leq k,i \leq d} = \pm\prod_{k<i}(y_k-y_i)$ is the Vandermonde determinant and $M_j$ is the determinant where the $j$th column of $V$ has been replaced by $(f(y_k))_{0\leq k\leq d}$.

    By assumption, we have $v(V) = \sum_{k<i}v(y_k-y_i) = \frac{d(d+1)}{2}n$.
    Factoring out suitable powers of $\pi$ from the columns of $M_j$, we find that
    \[
    v(M_j) \geq  \left( \frac{d(d+1)}{2} - j\right)n + \min_{0\leq k \leq d} v(f(y_k)).
    \]
    The result follows.
\end{proof}

The above single-variable interpolation lemma is used to prove the following result for our two-variable modular polynomials.

\begin{prop}\label{prop:general}
	Let $E_1$ and $E_2$ be elliptic curves over $K$ with good reduction and $j$-invariants $J_1 = j(E_1)$ and $J_2 = j(E_2)$.
    Let $N > 1$ with $v(N)=0$.
    Fix positive integers $n_1$ and $n_2$, and suppose that there exist elliptic curves $\CE_k/K$, $k=0,1,\ldots, \psi(N)$ satisfying the following conditions:
    \begin{enumerate}
        \item Each $\CE_k$ has good reduction;
        \item $v(j(\CE_k)-J_2) = v(j(\CE_k)-j(\CE_l)) = n_2$ for all $k\neq l$;
        \item For every $k$ and every elliptic curve $\tilde{\CE}_k$ linked to $\CE_k$ by a cyclic isogeny of degree $N$, we have
        \[
        v(j(\tilde{\CE}_k) - J_1) > 0 \Longrightarrow v(j(\tilde{\CE}_k) - J_1) \geq n_1.
        \]
    \end{enumerate}
    Then the coefficients of $\Phi_N(X+J_1, Y+J_2) = \sum_{0\leq i,j \leq \psi(N)}a_{i,j}X^iY^j\in A[X,Y]$ satisfy
    \[
    v(a_{i,j}) \geq n_1 C_{J_1,J_2}(N,\pi)-n_1 i-n_2 j,
    \]
    where $C_{J_1,J_2}(N,\pi) = \ord_X\big(\Phi_N(X+J_1,J_2) \bmod \pi\big).$
\end{prop}

\begin{proof}
	For ease of notation, let $C = C_{J_1,J_2}(N,\pi).$
Write 
\begin{align*}
    & \Phi_N(X+J_1,Y+J_2) = b_0(Y) + b_1(Y)X + \cdots + b_{\psi(N)-1}(Y)X^{\psi(N)-1} + X^{\psi(N)} \\
    & b_i(Y) = a_{i,0} + a_{i,1}Y + \cdots + a_{i,\psi(N)}Y^{\psi(N)}, \quad i=0,\ldots, \psi(N).
\end{align*}

For each $k$, let $y_k = j(\CE_k) - J_2$. The roots of $\Phi_N(X, y_k + J_2)$ are the $j$-invariants of elliptic curves $\CE_{k,m}$ linked to $\CE_k$ by a cyclic $N$-isogeny. Since $v(N)=0$, $\CE_k[N]$ is unramified and these elliptic curves and isogenies are all defined over $K$, since unramified extensions of $K$ correspond to extensions of the residue field $A/\pi$, which is algebraically closed.

By definition, $C$ of these roots $j(\CE_{k,m})$, $m=1,2,\ldots, C$, satisfy 
$v(j(\CE_{k,m})-J_1) >0$, so by assumption $v(j(\CE_{k,m} )-J_1) \geq n_1$, and the rest satisfy $v(j(\CE_{k,m})-J_1)=0$. 

The coefficients $b_i(y_k)$ of 
$\Phi_N(X+J_1, y_k+J_2)$ are symmetric forms in the roots $j(\CE_{k,m})-J_1$ which satisfy $v(j(\CE_{k,m})-J_1) \geq n_1$ for $m=1,2,\ldots, C$, so it follows that
\[
v(b_i(y_k)) \geq n_1(C -i), \quad i = 0, 1, \ldots, C.
\]

Now applying Lemma \ref{lem:interpolation} to the polynomials $b_i(Y)$ with the interpolation points $y_k$ completes the proof.
\end{proof}

\subsection{Deformations of elliptic curves.}

We start with the following result (\cite[Prop. 2.3]{GZ1985}):

\begin{prop}\label{GZ}
	Let $E_1$ and $E_2$ be elliptic curves over $K$ with good reduction and suppose $M\geq 1$ is the largest integer for which $E_1$ and $E_2$ are isomorphic over $A/\pi^M$. Then
	\begin{align*}
		v\big(j(E_1)-j(E_2)\big) & = \frac{1}{2}\sum_{m=1}^M\#\Isom_{A/\pi^m}(E_1,E_2) \\
		& = \frac{1}{2}\sum_{m=1}^M\#\Aut_{A/\pi^m}(E_1) \geq \frac{1}{2}\#\Aut_{A/\pi}(E_1). 
	\end{align*}
\end{prop}

Our goal now is to construct elliptic curves $\CE_k/K$ satisfying the hypotheses of Proposition~\ref{prop:general}. 

\begin{prop}\label{LSTG}
	Let $E/K$ be an elliptic curve with good reduction, and let $M\geq 1$ be a positive integer. Then there exist infinitely many elliptic curves $\CE_1, \CE_2, \CE_3, \ldots$ over $K$ such that each $\CE_k$ reduces to $E$ over $A/\pi^M$, but $E, \CE_1, \CE_2, \CE_3, \ldots$ are pairwise non-isomorphic over $A/\pi^{M+1}$. 
	
	In particular, $v(j(\CE_k) - j(E)) = v(j(\CE_k) - j(\CE_l)) = \frac{1}{2}\sum_{m=1}^M\#\Aut_{A/\pi^m}(E)$ for all $k\neq l$.  
\end{prop}

\begin{proof}
	By the Serre-Tate lifting theorem \cite[Thm. 3.3]{Conrad2004} and the Grothendieck existence theorem \cite[Thm 3.4]{Conrad2004}, the deformations of $E$
	are in bijection with the deformations of the $p$-divisible group $E[p^\infty]$. 
	
	When $E_{A/\pi}$ is supersingular, then $E[p^\infty] \cong \hat{E}$ is the formal group of $E$ which has height 2 and by \cite{LT1966} its deformations are given by a one-parameter family $\Gamma(t)$ with $t\in \pi A$. Let $t_0\in \pi A$ be the parameter for which $\hat{E} = \Gamma(t_0)$.
	Choose $t_k = t_0 + \pi^M \varepsilon_k$, for $\varepsilon_k\in A^*$ with pairwise distinct reductions modulo $\pi$. These correspond to infinitely many liftings 
	$\CE_k$ over $A$ which reduce to $E$ modulo $\pi^M$, but which are pairwise non-isomorphic over $A/\pi^{M+1}$ as deformations. Since the automorphism groups of elliptic curves are finite, we can choose an infinite subsequence of $\CE_k$'s which are pairwise non-isomorphic over $A/\pi^{M+1}$ as elliptic curves.
	
	When $E_{A/\pi}$ is ordinary, the deformations are parametrized by the Serre-Tate parameter $q\in 1 + \pi A$ (see \cite{Messing1972} or \cite{Meusers2017}). Let $q_0$ be the parameter associated to $E/A$ itself and again choose $q_k = q_0 + \pi^M \varepsilon_k$ for $\varepsilon_k\in A^*$ as above to obtain infinitely many suitable $\CE_k/A$.
	
	The claim on the $j$-invariant valuations now follows from Proposition \ref{GZ}.
\end{proof}

Theorem \ref{thm:bigprimes} now follows easily:

\begin{proof}[Proof of Theorem \ref{thm:bigprimes}]
	Let $E_1$ and $E_2$ be elliptic curves over $K$ with good reduction. By Proposition \ref{LSTG} with $M=1$ we obtain suitable deformations $\CE_k$ of $E_2$ with $v(j(\CE_k) - j(E_2)) = v(j(\CE_k) - j(\CE_l)) =  \frac{1}{2}\#\Aut_{A/\pi}(E_2)=n_2$ for $k\neq l$.
	
	Now let $\CE_k \to \tilde\CE_k$ be an isogeny of degree $N$. Since $v(N) = 0$, $\tilde\CE_k$ again has good reduction and Proposition \ref{GZ} gives $v(j(\tilde{\CE}_k) - j(E_1)) > 0 \Longrightarrow v(j(\tilde{\CE}_k) - j(E_1)) \geq \frac{1}{2}\#\Aut_{A/\pi}(E_1) = n_1$.  
	
	The hypotheses of Proposition \ref{prop:general} are thus satisfied and $v(a_{i,j}) \geq n_1 C_{J_1,J_2}(N,\pi) -n_1 i -n_2 j$ follows. It remains to notice that $\#\Aut_{A/\pi}(E_1) C_{J_1,J_2}(N,\pi) = \#\Aut_{A/\pi}(E_2) C_{J_2,J_1}(N,\pi)$, since both equal the total number of cyclic $N$-isogenies from $E_1$ to $E_2$ over $A/\pi$ (the second term counting the dual isogenies).
\end{proof}

\subsection{The case $E_1 = E_2$}

We need one more lemma.

\begin{lemma}\label{lem:minimal}
    Let $f : E \to E'$ be an isogeny of degree $N$ with $v(N)=0$ between elliptic curves over $A$ given by Weierstrass equations of the form (\ref{eq:WE}). Let $\omega_E = \frac{dx}{2y+a_1x+a_3}$ and $\omega_{E'}=\frac{dx}{2y + a'_1x+a'_3}$ be the invariant differentials of $E$ and $E'$ and suppose $f$ is normalized such that $f^*\omega_{E'} = \omega_E$. Then if the Weierstrass equation for $E$ is minimal, so is the Weierstrass equation for $E'$.  
\end{lemma}

\begin{proof}
    Let $\iota : E'\to \tilde{E}'$ be the isomorphism corresponding to a change of variables $(x,y) \mapsto (u^2x + r, u^3y + u^2sx + t)$ such that the Weierstrass equation for $\tilde{E}'$ is minimal. Then $(\iota\circ f)^*\tilde\omega_{\tilde{E}'} = u\,\omega_E$. But now $u\in A^*$ by \cite[Lemmas 4.3 and 4.4]{DD15}, so the Weierstrass equation for $E'$ is minimal, too.
\end{proof}

\begin{proof}[Proof of Theorem \ref{thm:smallprimes}.]
	We now assume the residue characteristic of $A/\pi$ is $p=2$ or $3$. 
    Let $E/K$ be an elliptic curve with good reduction, $J=j(E)\in \pi A$ and $v(N)=0$. Let
    \[
    n := \min\{ v(J - j(E')) \; : \; \exists f : E \to E' \; \text{cyclic isogeny of degree $N$}\}.
    \]
    Since $E$ (and thus any $E'$) has supersingular reduction, and $0$ is the only supersingular modulus in characteristic $2$ or $3$, it follows that $n  > 0$. Now by Proposition \ref{GZ} there exists $M \geq 1$ such that 
    $n = \frac{1}{2}\sum_{m=1}^M \#\Aut_{A/\pi^m}(E)$.
    
    Applying Proposition \ref{LSTG} with this $M$, we obtain deformations $\CE_k$ of $E$ satisfying $v(j(\CE_k)-J) = v(j(\CE_k) - j(\CE_l)) = n$ for all $k\neq l$. 
    
    Since $p\nmid N$, cyclic subgroups of $\CE_k[N]$ reduce to cyclic subgroups
    of $E[N]$. Thus a cyclic $N$-isogeny $\CE_k\to\tilde\CE_k$ reduces modulo
    $\pi^M$ to a cyclic $N$-isogeny from $E$ to some elliptic curve $E'$.
    By the definition of $n$ and Proposition \ref{GZ}, this $E'$ is isomorphic
    to $E$ over $A/\pi^M$. Hence $\tilde\CE_k$ is isomorphic to $\CE_k$ over
    $A/\pi^M$, and therefore
    \[
    v(j(\tilde\CE_k)-j(\CE_k))\geq n.
    \]

    It follows from Proposition \ref{prop:general} that $v(a_{i,j}) \geq n(C_J(N,\pi) -i -j)$. It remains to show that $n\geq n_v$, as claimed.

    We use V\'elu's explicit formulae for isogenies \cite{Velu1971}. Suppose $E$ has a minimal Weierstrass equation over $K$ as in (\ref{eq:WE}). Let $f : E \to E'$ be an isogeny with cyclic kernel $\ker f = C \subset E[N]$. If $N$ is even, then $C$ contains one point of order 2, which we denote $Q\in E[2]$. We partition $C$ into disjoint sets $C = R \cup (-R) \cup (C\cap E[2])$, so we have $\# R = \frac{N-2}{2}$ if $N$ is even, and $\#R = \frac{N-1}{2}$ if $N$ is odd.

    Then $E'$ is given by a minimal (by Lemma \ref{lem:minimal}) Weierstrass equation with coefficients $a_i'$, where
    \begin{align*}
        a'_1 &= a_1, & a'_2 &= a_2, & a'_3 &= a_3, \\
        a'_4 &= a_4 - 5t', & a'_6 &= a_6 -b_2t'-7w' \\
        c'_4 &= c_4 + 240t', & c'_6 &= c_6 + 504b_2t' + 6048w'.
    \end{align*}
    Here $t', w' \in A$ are given by (\ref{eq:tw_odd}) for $N$ odd and (\ref{eq:tw_even}) for $N$ even, where we make the substitutions 
    \begin{align*}
        x_1 & = \sum_{P\in R} x_P, & x_2&= \sum_{P\in R}x_P^2, & x_3 & = \sum_{P\in R}x_P^3 \\
        x_0 &= x_Q, & y_0 &= y_Q. 
    \end{align*}
    Since $E'/K$ has good reduction, we have $v(\Delta')=v(\Delta)=0$ and
    \[
    j(E') - j(E) = \frac{c'_4{}^3c_6^2 - c_4^3c'_6{}^2}{1728\Delta\Delta'},
    \]
    thus $v(j(E')-j(E)) = v(g')$, where $g'\in A$ is the polynomial $g$ from (\ref{eq:g}) with the variables specialized as above. It follows that 
    $v(g') \geq v(g) = n_v$.

    Finally, it remains to show that $n_v = v(g)$ only depends on the residue class of $N$ modulo 4 (when $p=2$) or 6 (when $p=3$). We have $v(t) \in [0,v(6)]$ because of the $6x_2$-term, and $v(w)\in [0,v(2)]$ because of the $10x_3$-term. The value of $N$ enters only via its parity and whether or not $\frac{N-1}{2}$ or $\frac{N-2}{2}$ is divisible by $p$. This concludes the proof of Theorem \ref{thm:smallprimes}.
\end{proof}

\begin{proof}[Proof of Theorem \ref{thm:rationalJ}]
    For each rational singular $j$-invariant $J$ and each prime $p$ we choose a globally minimal model $E/F$ 
    for an elliptic curve with $j(E)=J$ defined over a number field $F/\BQ$ for which
    $E$ has good reduction at the prime $\Fp$ of $F$ above $p$ and for which the ramification index $e_p = e(\Fp|p) = [F:\BQ]$ is minimal.

    Suitable models are found in the online database \cite{lmfdb}, except in the case $D=-27$ and $p=3$. In this case, one does find a model $E/\BQ(\sqrt{-3})$ with discriminant of norm $N_{\BQ(\sqrt{-3})/\BQ}(\Delta) = 3^{4}7^{6}$, and a suitable change of variables with $u=\sqrt[6]{-3}$ gives a global minimal model over $F=\BQ(\sqrt[6]{-3})$ with good reduction at the totally ramified prime above $3$. 

    Now we let $K = F_{\Fp}^{\ur}$ be the maximal unramified extension of the completion of $F$ at the prime $\Fp$ above $p$, normalized so that $v(p) = e_p$. Applying Theorems \ref{thm:bigprimes} and \ref{thm:smallprimes} to $E/K$ gives the result with $n_p = n_v/e_p$. When $p=2$ or $3$, it suffices to compute $v(g)$ for $N \leq 7$. The exceptional cases listed in Table \ref{tab:rationalJ} occur precisely when $e_p > 1$.
\end{proof}

\begin{proof}[Proof of Theorem \ref{thm:main}]
	This is the case $J=0$ and $D=-3$ of Theorem \ref{thm:rationalJ}.
	For $p=2,3,5$, Proposition \ref{prop:fullpsi} gives
	$C_0(N,p)=\psi(N)$. For $p\geq 11$, the condition
	$\left(\frac{-3}{p}\right)\neq 1$ is equivalent to
	$p\equiv 2 \bmod 3$, and the stated bound follows directly.
\end{proof}

\begin{rem}
    If $J\in\bar{\BQ}$ is any singular modulus, then we always find $n_v \geq 15$ when $p=2$ and $v(J)>0$. This follows from \cite[Corollary 2.5]{GZ1985}.
\end{rem}

\begin{rem}
    It is possible to give elementary proofs of Theorems \ref{thm:main} and \ref{thm:rationalJ} for each $J=j(E)$, which do not rely on the deformation theory of elliptic curves (Proposition \ref{LSTG}).
    
    For example, when $J=0$ and $p\neq 3$, one may define
    \[
    \CE_k : y^2 + y = x^3 + \varepsilon_k p x
    \]
    over $K=\BQ_p^\ur$ with $\varepsilon_k\in A^*$. 
    Direct calculations with V\'elu's formulae show that, for infinitely many choices of $\varepsilon_k\in A^*$, $\CE_k$ satisfies the hypotheses of Proposition \ref{prop:general} with $n_2=15$ and $n_p = 3$ when $p\geq 5$.
    
    In the case $J=0$ and $p=3$, let $K=\BQ_3^\ur(\sqrt{-3})$ and define
    \[
    \CE_k : y^2 = x^3 + \varepsilon_k\pi x^2 - \omega x
    \]
    over $K$, where $\omega = \frac{-1-\sqrt{-3}}{2}$ and $\pi = 1-\omega$. 
    Now the calculations are a little longer, but again one finds there are infinitely many choices of $\varepsilon_k\in A^*$ satisfying the hypotheses of Proposition \ref{prop:general} with $n_v = 6$; and when $N \equiv 1 \bmod 3$ one may choose $\varepsilon_k = 1 + \varepsilon'_k\pi\in A^*$ to obtain $n_v=9$. Theorem \ref{thm:main} then follows with $n_3 = n_v/e_3 = n_v/2$.
\end{rem}

\section*{Acknowledgements}  
The author would like to thank Fabien Pazuki, Ren\'e Schoof and John Voight for insightful questions and useful discussions and Haiyang Wang for sending him advance versions of \cite{Wang-prep}.

The author acknowledges the help of AI tools (Claude Opus 4.8 and ChatGPT 5.5) in suggesting minor clarifications of some arguments.

\bibliographystyle{plain}
\bibliography{Divisibility}

\end{document}